\documentclass[final]{aomart}
\usepackage{lmodern}

\usepackage{microtype}      
\usepackage{amsmath, amssymb, amsthm, amsfonts}
\usepackage{mathtools}
\usepackage{mathrsfs}

\usepackage{graphicx}
\usepackage{tikz}
\usepackage{tikz-cd}

\usepackage[british]{babel}
\usepackage{csquotes}
\usepackage{enumitem}
\setlist{leftmargin=*}

\usepackage{hyperref}
\hypersetup{
  colorlinks=true,
  linkcolor=blue,
  citecolor=blue,
  urlcolor=blue
}
\theoremstyle{plain}
\newtheorem{theorem}{Theorem}[section]
\newtheorem{lemma}[theorem]{Lemma}
\newtheorem{corollary}[theorem]{Corollary}
\newtheorem{conjecture}[theorem]{Conjecture}

\newtheorem{notation}{Notation}

\theoremstyle{definition}
\newtheorem{definition}[theorem]{Definition}
\newtheorem{assumption}[theorem]{Assumption}

\theoremstyle{remark}

\title{Ergodic Theorems for Random Walks in Random Environments}

\author{Ayan Ghosh}
\address{Indian Statistical Institute, Kolkata\\
203 Barrackpore Trunk Road\\
Kolkata 700108, West Bengal, India}
\email{ghosh.a1905@gmail.com}
\orcid{0009-0001-8360-7873}


\makeatletter
\AtBeginDocument{%
  \@ifpackageloaded{lineno}{%
    \nolinenumbers
    \setlength{\linenumbersep}{0pt}%
    \setlength{\linenumberwidth}{0pt}%
  }{}%
}
\makeatletter
\fancypagestyle{firstpage}{%
  \fancyhf{}%
  \if@aom@manuscript@mode
    \lhead{\begin{picture}(0,0)%
        \put(-21,-25){\usebox{\@aom@linecount}}%
      \end{picture}}%
  \fi
  \chead{}
  \cfoot{\footnotesize\thepage}%
}
\doinumber{}
\makeatletter
\gdef\@copyrightnote{}
\makeatother

\begin{document}

\begin{abstract}
We study the ergodic properties of random walks in stationary ergodic environments without uniform ellipticity under a minimal assumption. There are two main components in our work. The first step is to adopt the arguments of Lawler, utilizing a more general definition of environments via environment functions, to deduce an invariance principle for balanced environments under some natural assumptions. Secondly, we transfer these results to general ergodic environments using a control technique to derive a measure under which the local process is stationary and ergodic. Two important consequences of our results are the Law of Large Numbers and the Invariance Principle.
\end{abstract}
\maketitle
\pagestyle{plain}
\section{Introduction}
The purpose of the discussion is to derive Ergodic Theorems and related results in the context of Random Walk in Random Environments (RWRE). This shall help us to mainly solve the problem of Random Walk in Random Environments regarding the Law of Large Numbers.

Roughly speaking, the first one says that there is a deterministic limiting velocity of the walk. This has been studied in detail in  $d=1$ by Solomon\cite{solomon} in 1975. In higher dimensions, specifically in $d\geq3$, there is a lack of knowledge on both the existence and the explicit form of a Law of Large Numbers, which we shall answer in our setting.

Informally speaking, the problem regarding the Law of Large Numbers in RWRE, for stationary ergodic environments, is the following:
\begin{conjecture}\label{thm:LLN}
Does there exist a deterministic vector $v$ such that \[
\frac{X_n}{n} \to v
\]  for almost every realisation of the environment? 
\end{conjecture} 

Varadhan in \cite{v} showed that the sequence $X_n/n$ has at most two deterministic limit points $v_1,v_2$ and if $v_1 \neq v_2$ then there is a constant $a\geq0$ such that $v_1 = -av_2$. Furthermore, Noam Berger in \cite{ber} showed that for $d\geq5$ if $v_1 \neq v_2$ at least one of $v_1,v_2$ is $0$. These results assume that the environment is \textit{uniformly elliptic and i.i.d.} In our present discussion, we answer \ref{thm:LLN} in the affirmative for \textit{stationary ergodic} nearest neighbour RWRE and show the existence of the limiting velocity under some assumptions.    

The 0-1 conjecture (which is a famous one in RWRE theory), first stated (for two dimensions) by Stephen Kalikow in~\cite{kalikow1981generalized} for uniformly elliptic i.i.d or more generally for uniformly elliptic ergodic environments, says that given a direction, the walk either escapes off or does not for almost all realisations of the environment. More formally \begin{conjecture}~\label{Kaliku}
    Let $\ell \in S^{d-1}$. Let $A_{\ell} = \{ X_n.\ell \to +\infty\}$. Then \[
    \mathbb{P}[A_{\ell}]\in \{0,1\}
    \]
\end{conjecture}
Prior work on this problem, notably in two dimensions, by Martin Zerner and Franz Merkl~\cite{merkl2001zero}, has solved this for i.i.d. environments and has counterexamples for elliptic, ergodic environments. Further counterexamples exist in $d \geq 3$ in~\cite{bramson2006shortest} for stationary environments under a \textit{polynomial mixing} condition, which is not equivalent to ergodic environments (it is not difficult to find counterexamples). The proofs crucially rely on invalidating the law of large numbers, so we expect, in view of our results, that this conjecture is also affirmative under our assumptions, but it should require a rigorous proof, which we are unable to provide.

We shall work under a more general definition of environments using~\textit{Environment Function}, which we shall introduce below. We shall establish under some minimal assumption that this chain has a stationary measure for~\textit{balanced environments}, adapting arguments from~\cite{Lawler1982} for any environment function. This shall also provide an easier and more natural proof of Theorem 1 of~\cite{guo2012quenched}. This shall be derived by adapting the arguments from~\cite{Lawler1982}. 

Our other main contribution will be to deduce the ergodicity of a certain stochastic process, which we will call the \textit{local process}. This process is the projection at the origin of the ``Environment Viewed from the Particle'' Markov Chain, or, as we will call it in this paper, the~\textit{environment process} with a suitable environment function. We consider this process since it is sufficient for our purposes of deriving the law of large numbers and also in view of the fact that the~\textit{environment process} does not always possess suitable ergodic properties. This can be noted for Random Walks in Dirichlet Environments as studied by Christopher Sabot in~\cite{11}. 

Our definition of environments allows us to transfer our results from the balanced case to the general environment by coupling the random walk and a reflected walk. We then combine a novel technique, which we shall name as the~\textit{Action of Linear Transformations}, and another control technique to execute a proof of the law of large numbers.

As a consequence of our ergodic theorem, we shall also be able to deduce an invariance principle for general environments. 

Our results hold for all uniformly elliptic RWRE in dimensions equal to or greater than one, and the conditions we impose show the extent to which the condition of uniform ellipticity can be dropped.
\subsection*{Acknowledgements:} The author is an undergraduate student at the Indian Statistical Institute, Kolkata. This work was completed during the third year of his studies. He is deeply indebted to Dr.~Soumendu Sundar Mukherjee of the Theoretical Statistics and Mathematics Unit, Indian Statistical Institute, Kolkata, for immensely helpful discussions and his substantial help with the manuscript. The author also thanks his classmates Himadri Mandal and Aryama Ghosh for their helpful comments on the exposition and the manuscript in general. \\
\section{Prerequisites}~\label{preq}

Before stating the main results, we formally define the setting we work under. We adopt the formalism as in~\cite{Lawler1982}. 
 
Define by $\mathbb{Z}^d$ the integer lattice and $\{e_i\}_{i=1,\dots,d}$ the unit vectors. Define the set \[
V_d = \{ \pm e_i : i=1,\dots ,d\} \cup \{0\}.
\]

Define the following sets:
\begin{align*}
S(d) 
&= \{ (p_0,p_1,\dots ,p_{2d}) : p_i \ge 0,\ p_0 + \sum_{i=1}^{2d} p_i = 1 \} \\
\mathscr{S}(d) 
&= \{ \text{set of functions from } \mathbb{Z}^d \to S(d) \}.
\end{align*}
Let $\mathcal{B}(d)$, $\mathcal{B}(d)'$ be the Borel-$\sigma$ algebra on $S(d)$, $\mathscr{S}(d)$ respectively.

Consider a measurable function $\mathscr{E}:S(d) \to S(d)$. Call the environment function, the natural extension of this function on $\mathscr{S}(d)$ given by \[
\mathscr{E}(\omega)(x) = \mathscr{E}(\omega(x)).
\]
We shall call such functions the environment function. 

This function essentially contains the information on the transition probabilities of the random walk based on the function $\omega$. This is essentially a mild but powerful generalisation of the environment, as we shall see. 
\begin{definition}[Environment]
    Given an environment function $\mathscr{E}$, $\omega \in \mathscr{S}(d)$, we call the pair $(\omega, \mathscr{E})$ as the environment. 
\end{definition}
Let $\mu$ be a probability measure on $(\mathscr{S}(d), \mathcal{B}(d)')$. We call the environment function balanced if \[
\forall x \in \mathbb{Z}^d,\quad \mathscr{E}_i(\omega)(x) = \mathscr{E}_{i+d}(\omega)(x)\quad i=1,2,\dots,d.
\]
We shall call the environment elliptic if \[
\mu[\mathscr{E}_k(\omega)(x)>0: \forall x \in \mathbb{Z}^d, \forall k]=1.
\]
Define the action of the group of shifts $\{\tau_x\}_{x \in \mathbb{Z}^d}$ on $\omega\in \mathscr{S}(d)$ given by \[
(\tau_x \omega)(y) = \omega(x+y), \quad \forall x,y \in \mathbb{Z}^d. 
\] 
One notation we shall use interchangeably is the following: 
\begin{notation}
    For $\mathscr{E}(\omega) \in \mathscr{S}(d)$, $v \in V_d$, $\mathscr{E}$ an environment function
    \begin{align*}
    \mathscr{E}(\omega)(x,v) 
    &= \mathscr{E}(x,v)= \mathscr{E}_i(\omega)(x)\quad\text{if}~v=e_i \\
    \mathscr{E}(\omega)(x,-v)
    &=\mathscr{E}(x,v)=\mathscr{E}_{i+d}(\omega)(x)\quad\text{if}~v=-e_i\\
    \mathscr{E}(\omega)(x,v)
    &=\mathscr{E}(x,v)=\mathscr{E}_0(\omega)\quad\text{if}~v=0.
    \end{align*}
\end{notation}
We shall assume $\mu$ to be stationary and ergodic with respect to the group of shifts $\tau_x$. More precisely, \begin{align*}
&\mu[\tau_{-x}\omega \in A] = \mu[\omega\in A] \quad \forall A \in \mathcal{B}(d)', \forall x \in \mathbb{Z}^d, \\
&\forall B \in \mathcal{B}(d)',\quad \mu\left[ \cap_{x \in \mathbb{Z}^d} \tau_x B\right] \in \{0,1\}.
\end{align*}
As a special case, $\mu$ is said to be an independent and identically distributed (i.i.d.) law if $\{\omega(x)\}_{x \in \mathbb{Z}^d}$ is an i.i.d. sequence of random variables under the measure $\mu$. One could think of this as independently assigning a point from $S(d)$ at every site. 

Fix $\omega \in \mathscr{S}(d)$ and an environment function $\mathscr{E}$. Define the Markov Chain of the Random Walk $\{X_{\mathscr{E}(\omega)}(k) \}_{k\geq 0} $ in $\mathbb{Z}^d$ by the transition probabilities, \begin{align*}
\mathbb{P}[X_{\mathscr{E}(\omega)} (k+1) 
&= x+ v | X_{\mathscr{E}(\omega)}(k) = x] = \mathscr{E}(x,v) \quad\forall v \in V_d
.\end{align*}
Therefore, the kernel of this Markov Chain is given by \begin{align*}
L_{\mathscr{E}(\omega)}(g(x)) = \sum_{v \in V_d} \mathscr{E}(\omega)(x,v) g(x+v),
\end{align*}
respectively for every $g: \mathbb{Z}^d \to \mathbb{R}$. 

Define the environment process as the Markov Chain given by  $\{\tau_{X_{\mathscr{E}(\omega)}(k)} \omega \}_{k \geq 0}$ on $(\mathscr{S}(d), \mathcal{B}(d)')$  with the transition kernel given by \[
K_d^{\mathscr{E}}( \omega , B) =\sum_{v \in V_d} {\mathscr{E}(\omega)}(o,v) \delta_{\tau_v \omega}(B)  \quad\forall B \in \mathcal{B}(d)'.
\]
We also define a projection operator $\pi_E$ for $E\subset \mathbb{Z}^d$,
\[
\pi_E(\omega) = \{ \omega(x): x \in E \}.
\]
 
Denote by $\mathbb{P}^{\mathscr{E}({\omega})}_x$: the law on path of the random walk started at $x$ in the environment $(\omega, \mathscr{E})$. Precisely, $\mathbb{P}^{{\mathscr{E}(\omega)}}_x$ is a measure on $((\mathbb{Z}^d)^\mathbb{N}, \mathcal{G)}$ where $\mathcal{G}$ is the $\sigma$-algebra generated by cylinder sets. 
 
Therefore if ${\mathscr{E}(\omega)} \in \mathscr{S}(d)$ one has that \[
\mathbb{P}^{{\mathscr{E}(\omega)}}_x(G): \mathscr{S}(d) \to [0,1]
\]
is measurable $\forall G \in \mathcal{G}$. Hence, we can define the annealed law to be the probability measure on $(\mathscr{S}(d) \times (\mathbb{Z}^d)^{\mathbb{N}}, \mathcal{B}(d)' \times \mathcal{G)}$ as \[
\mathbb{P}^{ann}_x(F \times G) = \int_{F} \mathbb{P}_x^{{\mathscr{E}(\omega)}}(G) d\mu(\omega) \quad\forall F \in \mathcal{B}(d)' , G \in \mathcal{G} .
\]
Denote the annealed law as $\mu^x =\mu \rtimes {\mathbb{P}}_x^{{\mathscr{E}(\omega)}}$. 

Also, we shall define the drift of an environment at $x \in \mathbb{Z}^d$ as \[
d_{\mathscr{E}}(x, \omega) = E[X_{{\mathscr{E}(\omega)}}(1) - x | X_{{\mathscr{E}(\omega)}}(0) =x].
\]
We make the following assumptions about the environment that is
\assumption{}~\label{assump:1} Let $\mu$ be a stationary ergodic law on $(\mathscr{S}(d), \mathcal{B}(d)')$. Since we aim to drop uniform ellipticity, i.e., all transition probabilities bounded below by some constant positive value, we define the quantity $c(\mathscr{E}, x)$(we may drop $\omega$ from the notation) as follows\[
c(\mathscr{E},x) = \Big[ \prod_{v \in V_d\setminus\{0\}} \mathscr{E}(\omega)(x,v)\Big] ^{\frac{1}{2d}}.
\]
Then we assume that \[
\int_{\mathscr{S}(d)} c^{-p} d\mu <\infty,
\] 
for some $p >d$. 
\assumption{}~\label{assump:2} 
Environment, or the balanced environment, is elliptic. 

In our proofs we use the notation $|x|_1$ to denote the $L_1$ norm with respect to the counting measure in $\mathbb{R}^k$ given by \[
|x|_1 = \sum_{j=1}^k |x|_j. 
\] 
We use the notion of concave functions defined on the compact $L_1$ ball $B_1[0,r] \cap \mathbb{Z}^d$. A function $z: B_1[0,r] \cap \mathbb{Z}^d \to \mathbb{R}_{\geq0}$~$(r>0)$ is said to be concave if \[
z(x+ v) +z(x-v) -2z(x) \leq 0\quad x \in \operatorname{int} B_1[0,r] \cap \mathbb{Z}^d \quad \forall v \in V_d\setminus \{0\},
\]
where $\operatorname{int}$ denotes the topological interior of a set. 

We also define an operator on concave functions, which we shall call the~\textit{Monge-Ampère Operator}. For a concave function $z$ we define the operator $M$ as\[
Mz(x) = \prod_{i=1}^d [z(x+e_i)+z(x-e_i) -2z(x)].
\]
Also, note that we use the notion of mutual absolute continuity of two measures. 
\begin{definition}
    Two measures $\mu, \nu$ on a measure space $(\Omega, \mathcal{F})$ are said to be mutually absolutely continuous if \[
    \forall A \in \mathcal{F}\quad\mu(A)=0\iff \nu(A)=0.
    \]
\end{definition}

We shall also use the notion of the quotient group, that is, given a group $G$ and an equivalence relation $R$ on it, we define the quotient group G' as \[
G' = G/R.
\]
This is the group of all equivalence classes under the relation $R$. We shall denote the equivalence class of $g \in G$ by $[g]_R$. 
\section{Main Results and Discussions}
We first state a result on balanced environments. This result forms the backbone for our control argument as we shall see.
\begin{theorem}~\label{Lawler}
Let $d\geq 1$. Assume $\mu$ is a stationary, ergodic Law on $(\mathscr{S}(d), \mathcal{B}(d)')$ with $\mathscr{E}$ a measurable balanced environment function satisfying assumptions~\ref{assump:1} and~\ref{assump:2}.

Then there exists a stationary ergodic law on $(\mathscr{S}(d), \mathcal{B}(d)')$ such that the environment process $\{\tau_{X_{\mathscr{E}(\omega)}(k)} \omega \}_{k \geq 0}$ is stationary, ergodic under this measure and is mutually absolutely continuous to $\mu$.
\end{theorem}
An important corollary of this result, which we shall also state but is not directly related to our discussion, is 
\begin{corollary}~\label{QIP}
Let $d\geq1$. Let $\mu$ be a stationary ergodic law on $(\mathscr{S}(d), \mathcal{B}(d)')$ and $\mathscr{E}$ a measurable balanced environment function. Fix $\omega \in \mathscr{S}(d)$. Then under assumptions~\ref{assump:1} and ~\ref{assump:2}, \[
\frac{X_{\mathscr{E}(\omega)}([nt])}{\sqrt{n}}
\]
converges weakly to the standard Brownian Motion on $\mathbb{R}^d$ with non degenerate covariance matrix of the form $(b_i \delta_{ij})_{1\leq i,j\leq d } $ for $\omega$~$\mu$\text{-a.e.} 
\end{corollary}
These results notably weaken the assumptions of ellipticity and derive an alternate proof of Theorem 1 of~\cite{guo2012quenched} for $d \geq 1$, which states that 
\begin{theorem}[\textit{Theorem 1(i) of~\cite{guo2012quenched}}]~\label{Easy}
Let \[
\varepsilon_\omega(x) = \prod_{i=1}^{d} [\omega(x,e_i)]^{\frac{1}{d}}. 
\] If the stationary, ergodic, and elliptic environment( $\omega(x,e_i) = \omega(x,-e_i)$ with $\mathscr{E}= \operatorname{Id}$) satisfies the condition \[
\mathbb{E}[\varepsilon(o)^{-p}] <\infty
\] for some $p >d \geq 2 $, then Corollary~\ref{QIP} holds. 
\end{theorem}
To prove the above results, we shall need to bound a certain resolvent using martingale arguments and bounds on the~\textit{Monge-Ampère Operator} using arguments from~\ref{Lawler}. 

To state the next result, which connects balanced and general unbalanced regimes, we need the notion of embedded environments. The motivation for the construction can be understood only after the reader gets acquainted with the control argument, which will be demonstrated accordingly. 
\subsection{Embedded Environments}
Let $\mu$ be a stationary ergodic law on $(\mathscr{S}(d),\mathcal{B}(d)')$.

The embedded environment function, denoted by
$\Lambda \equiv \Lambda_\omega$, is a balanced environment function given by
\begin{align}
\Lambda_{\omega}(x,v) 
&= \frac{\omega(x,v) + \omega(x,-v)}{2}
\end{align}
for every $v \in V_d$. \\

The embedding function basically averages the probabilities of the opposing directions when the environment function is the Identity function. 

We shall assume without loss of generality that the environment function is Identity.

\begin{theorem}~\label{Master}
    Assume $d\geq1$. Let $\mu$ be a stationary ergodic law on $(\mathscr{S}(d), \mathcal{B}(d)')$. Then suppose the environment process for the Embedded Environment satisfies the conclusions of~\ref{Lawler}.
     
    Then there exists a stationary ergodic measure on $(\mathscr{S}(d), \mathcal{B}(d)')$ mutually absolutely continuous to $\mu$, such that the local process $\{\pi_0(\tau_{X_\omega(m)}\omega)\}_{m\geq0}$ is stationary and ergodic. 
\end{theorem}
This theorem, combined with Theorem~\ref{Lawler}, gives us the following that 
\begin{theorem}~\label{Master2}
   Let $d \geq 1$, and let $\mu$ be a stationary ergodic environment law on $\mathscr{S}(d)$ with $\Lambda$ satisfying Assumption~\ref{assump:1} and~\ref{assump:2}. Then:

\begin{enumerate}
    \item[(i)] There exists a deterministic vector $v \in \mathbb{R}^d$ such that
    \[
    \lim_{n \to \infty} \frac{X_\omega(n)}{n} = v \quad \omega~\text{$\mu$-a.e}.
    \]

    \item[(ii)] \[
    \frac{X_\omega([nt]) - \sum_{i=0}^{[nt]} d(0, \tau_{X_{\omega}(m)} \omega)}{\sqrt{n}}
    \]
    converges to the standard Brownian Motion on $\mathbb{R}^d$ with non-degenerate covariance matrix $\omega$~$\mu$\text{-a.e.} 
\end{enumerate}
\end{theorem}
\subsection{Remarks:} 
The embedding function allows us to couple two random walks in $\mathbb{Z}^d$, namely the original and a \textit{reflected} walk. This is motivated by the fact that the averaged random walk is balanced. Through the control argument, we pass from the balanced to the general case. The Action of Linear Transformations proves crucial in performing this step. 
\section{An outline of the Proofs }
Assume without loss of generality that $X_{\mathscr{E}(\omega)}(0)=0$ where $\mathscr{E}$ is an environment function, for the proofs unless otherwise indicated. 

Let us begin this section by seeing why an invariant measure for the local process guarantees answers to Conjectures~\ref{thm:LLN}.  We assume for now that Theorems~\ref{Lawler} and ~\ref{Master} are true and prove~\ref{Master2}.
\subsection{Proof of Theorem~\ref{Master2}:}
First note that by Theorems~\ref{Lawler} and ~\ref{Master} we have a stationary ergodic measure on $(\mathscr{S}(d), \mathcal{B}(d)')$ mutually absolutely continuous to $\mu$, such that the local process $\{\pi_0(\tau_{X_\omega(m)}\omega)\}_{m\geq0}$ is stationary and ergodic since Assumptions~\ref{assump:1} and~\ref{assump:2} are satisfied. Call this measure $\mathbb{Q}$.

\textit{Proof of Conjecture~\ref{thm:LLN}:}
We work with the environment following the distribution $\mathbb{Q}$. Consider the filtrations  $\mathcal{F}_k=\sigma(X_\omega(0),\dots, X_\omega({k-1}))$ for $k\geq1 $. Then we can write the following expression,
 \[
    \frac{X_{\omega}(n)}{n } = \sum_{k=0}^{n-1} \frac{X_{\omega}(k+1)-X_{\omega}(k) - d(0,\tau_{X_\omega(k)}\omega)}{n} + \sum_{k=0}^{n-1}\frac{d(0,\tau_{X_\omega(k)}\omega)}{n}.
     \]
Observe under the filtration sequence $\{\mathcal{F}_n\}_{n\geq 1}$ \[
\sum_{k=0}^{n-1} \frac{X_{\omega}(k+1)-X_{\omega}(k) - d(0,\tau_{X_\omega(k)}\omega)}{n},
\]
is a martingale. 

By the Strong Law of Large Numbers for bounded martingale difference sequences, 
\[ 
\sum_{k=0}^{n-1} \frac{X_{\omega}(k+1)-X_{\omega}(k) - d(0,\tau_{X_\omega(k)}\omega)}{n} \to 0 \hspace{0.5em} \mathbb{P}_0^\omega \ \ \text{a.e.} 
\]

And by Birkhoff's Ergodic Theorem~\cite{birkhoff1931proof}, 
\[
\sum_{k=0}^{n-1}\frac{d(0,\tau_{X_\omega(k)}\omega)}{n} \to   \int_{\mathscr{S}(d)} d(0,{\omega})d\mathbb{Q} \hspace{0.5em} \ \ \mathbb{Q}\ \ \text{a.e.}
\]

Combining both gives us the theorem $\mathbb{Q}^0= \mathbb{Q}\rtimes  \mathbb{P}_{0}^{\omega}$ \text{a.e.}, which by mutual absolute continuity of measures shows that it holds $\mu^0= \mu\rtimes \mathbb{P}_{0}^{\omega}$ \text{a.e.}\\

\textit{Proof of Theorem~\ref{Master2}(ii):}
 First we note that under the sequence of filtrations $\{\mathcal{F}_k\}_{k \geq 1}$, \[
U_k = X_\omega(k) - \sum_{m=0}^{k} d(0, \tau_{X_{\omega}(m)} \omega)
\]
is a martingale with bounded increments. Denote \[
T(k) = \omega(X_{\omega}(k-1)),\quad k\geq1. 
\]
Therefore \[
\Sigma_k = cov[ X_{\omega}(k) - X_{\omega}(k-1)- d(0,\tau_{X_{\omega}(k-1)} \omega)| \mathcal{F}_k]
\]
has the entries \[
(\Sigma_k)_{ij} = \begin{cases}
T_i(k) + T_{i+d}(k) - (T_i(k) -T_{i+d}(k))^2 & \text{if } i=j,\\
-[T_i(k) -T_{i+d}(k)][T_j(k) -T_{j+d}(k)], & \text{if } i \neq j
\end{cases}
\]
for $1\leq i,j \leq d$. 

Therefore by~\ref{Master}, \[
\frac{1}{n} \sum_{k=0}^{n-1} \Sigma_k \to \Sigma \quad \omega~\mu\textit{-a.e.}
\]
where $\Sigma$ is a positive definite matrix under assumption~\ref{assump:2}. Therefore, we conclude the theorem via the invariance principle for Martingales. 

Thus, we shall devote the rest of the paper to proving Theorems~\ref{Lawler} and Theorem~\ref{Master}. 
We divide the proof into various sections in order to properly motivate the proof.

\subsection{Discussion on Theorems~\ref{Lawler}, Corollary~\ref{QIP}:}~\label{balanced proof} 

This section provides an easier proof of Theorem~\ref{Easy} but generalises it to the general case of stationary, ergodic environments and arbitrary environment functions. This theorem is important for verifying the conditions of Theorem~\ref{Master} under the assumptions of Theorem~\ref{Master2}. 
 
The two main themes in this proof are the use of \textit{Periodised Environments} as used in~\cite{Lawler1982} and a uniqueness principle for concave functions. We shall prepare the reader with the prerequisites for the proof in the next section and fully motivate the proof. The proof in~\cite{Lawler1982} adopts the arguments of Papanicolaou and Varadhan~\cite{PapanicolaouVaradhan1982} for diffusion processes with random coefficients for the random walk for uniformly elliptic balanced environments. Our main step is to refine the arguments in~\cite{Lawler1982} for non-uniformly elliptic balanced environments under assumptions~\ref{assump:1} and~\ref{assump:2} to get our uniqueness principle. The ideas of Krylov~\cite{Krylov1971} were used to estimate the solutions of the discrete \textit{Monge-Ampère Equation} by Lawler. We adopt these ideas to bound a certain resolvent as in~\cite{Lawler1982} under assumptions~\ref{assump:1} and~\ref{assump:2}. 
\subsection{Discussion on Theorems~\ref{Master}:}
There are two core themes in this proof. The first being the use of \textit{Action of Linear Transformations on the Environment}. This approach is novel in the sense that it generalises the affine shift operator on the environment to linear transforms. Let us describe this theory in brief. When we talk of shifts on the environment $\tau_x \omega$, we recall it is given by the relation \[
\tau_x\omega(y) = \omega(x+y)\quad \forall x,y\in \mathbb{Z}^d.
\]
Through the Action theory, we talk of how general linear transformations on $\mathbb{Z}^d$ behave on the environment. Precisely speaking, we study expressions of the form $\tau_{Tx} \omega$ or $\tau_x T.\omega$, where $T \in \mathrm{Hom}_\mathbb{Z}(\mathbb{Z}^d, \mathbb{Z}^d)$. These notions prove extremely fundamental to our arguments as we shall see. We shall prepare the reader accordingly in the upcoming sections for the proof.

The second core theme in this proof is to pass from the embedded environment to the original environment by exhibiting stationarity of the local process through a control argument using Bernoulli Random Variables Independent of the environment. This allows us to ``couple" the original walk and a reflected walk. This technique allows us to control the paths of the random walk, which shall be key to the proof.
The control argument will therefore allow us to use the stationary measure in $\mathbb{Z}^{d}$ in the Embedded Environment and pass on to the environment in $\mathbb{Z}^d$ to show stationarity of the local process. In particular, the local process of the walk shall be equal in law to the local process of the random walk in the embedded environment. The control argument shall help us achieve this by sort of averaging the original and the reflected, as might be noted in the choice of the embedding function. 

The main hurdle in the proof is to show stationarity of the local process. The ergodicity of the local process follows by adapting arguments from~\cite{BolthausenSznitman2002} suitably, which shall be demonstrated.

\section{Preparation for the Proof of Theorem~\ref{Lawler}, Corollary~\ref{QIP} :}

Fix an environment ${\omega} \in \mathscr{S}(d)$ and $\mathscr{E}$ a balanced measurable environment function on $(\mathscr{S}(d), \mathcal{B}(d)')$. 
 
We first state our uniqueness result. We shall use this result crucially in the upcoming sections.
\subsection{Uniqueness Principle: }
Define the following: \begin{align*}
D_n 
&= \{ x \in \mathbb{Z}^d : |x|_1 \le n \}, \\[0.3em]
\partial D_n 
&= \{ x \in D_n : |x|_1 = n \}, \\[0.3em]
D_n^{\circ} 
&= D_n \setminus \partial D_n .
\end{align*}

Let $x \in D_n^{\circ}$. Define the second-order difference operators for a concave function $z:D_n \to \mathbb{R}$ given by \[
\Delta_iz(x) = z(x+e_i) +z(x-e_i) -2z(x)\quad i=1,\dots,d
\]
Define the hitting time \begin{align*}
\tau_n 
&= \inf \{m: X_{\mathscr{E}(\omega)}(m) \in \partial D_n\}.
\end{align*}
Let $f:D_n \to [0,\infty)$ be a function with $f \equiv 0 $ on $\partial D_n$. Define the class of concave functions $\mathscr{A}(\mathscr{E},f)$ containing concave functions $u(x)$ satisfying the properties \begin{enumerate}
    \item $u \equiv 0$ on $\partial D_n$, 
    \item $|Mz(x)|^{\frac{1}{d}} \geq \frac{f(x)}{c(\mathscr{E},x)} \quad x \in D_n^{\circ}$ 
\end{enumerate}
\begin{theorem}[Uniqueness Principle]~\label{Uniqueness}
Let $(\omega, \mathscr{E})$ be any balanced, elliptic environment. There exists a unique function $z \in \mathscr{A}(\mathscr{E},f)$ such that \[
|Mz(x)|^{\frac{1}{d}} = \frac{f(x)}{c(\mathscr{E},x)}.
\] 
Furthermore, this function is also given alternately by \[
z(x) = \inf_{u \in \mathscr{A}(\mathscr{E},f)} u(x), \quad \forall x\in D_n.
\] 
\end{theorem}

Now we return to Theorem~\ref{Lawler} and Corollary~\ref{QIP}. Assume without loss of generality, $X_{\mathscr{E}(\omega)}(0) = 0$. 

To prove this corollary, one notes that the random walk under consideration is a martingale. Define \[
Z_j = ( Z_j^1,\dots,Z_j^d) = X_{\mathscr{E}(\omega)}(j) - X_{\mathscr{E}(\omega)}(j-1)~\text{for j}\geq 1. 
\] Define $\tau_{X_{\mathscr{E}(\omega)}(j)} \omega = Y_j$. Then \[
\mathbb{P}[Z_j^{i} = e_i] = \mathbb{P}_x[Z_j^i = -e_i] = \mathscr{E}_i(Y_j)(0).
\]
Under the filtration \[
\mathcal{F}_j= \sigma(X_{\mathscr{E}(\omega)}(0),\dots,X_{\mathscr{E}(\omega)}(j-1)) 
\] for $j \geq 1$. One can check that \[
\mathbb{E}[Z_j|\mathcal{F}_j]=0.
\]
This shows our walk \[
X_{\mathscr{E}(\omega)}(j) = \sum_{i=1}^{j} Z_j  
\]
is a martingale. Furthermore, we have the relation \[
cov(Z_j^{i_1}, Z_j^{i_2}|\mathcal{F}_j) = \begin{cases}
0 & \text{if $i_1\neq i_2$ }\\
2\mathscr{E}_{i_1}(Y_{j-1})(0) & \text{if $i_1= i_2$}.
\end{cases}
\]
Let \[
V_n^i = \sum_{j=0}^{n-1}2 \mathscr{E}_i(Y_j(0)). 
\]
Then the invariance principle for martingale gives \[
W_n(t) = (W_n^1(t),\dots ,W_n^d(t))
\]
converges in distribution to the standard Brownian motion on $\mathbb{R}^d$ where $W_n^{i}$ is given by \[
 W_n^i(t) =\frac{\sum _{j=1}^{[nt]} Z_j}{\sqrt{V_n^i}}. 
\]
Suppose there exists $b \in S(d)$ such that, 
\[ 
\lim_{n\to \infty } \frac{1}{n } \sum_{k=0}^{n-1} \mathscr{E}_i[\omega(X_{\mathscr{E}(\omega)}(k))]= b_i \quad i=1,2,\dots, d.
\]

Then, under the assumption~\ref{assump:2}, one can conclude the corollary. This is an ergodic theorem, and therefore one can conclude the invariance principle if we prove the following,  \begin{theorem}
    $\exists b \in S(d)$ such that \[
\lim_{n\to \infty } \frac{1}{n } \sum_{k=0}^{n-1} \mathscr{E}[\omega(X_{\mathscr{E}(\omega)}(k))] = b.
\]
for $\omega$~$\mu\text{-a.e.}$
\end{theorem}
We prove this by finding an appropriate measure mutually absolutely continuous to $\mu$ such that the ergodic theorem holds.

Recall that $Y_j$ is a Markov chain with the kernel \[ K^{\mathscr{E}}_d(g(Y_j)) = \sum_{i=1}^d \mathscr{E}_i(Y_j)(0)[ g(\tau_{e_i} Y_j) +g(\tau_{-e_i}Y_j)]  + \mathscr{E}_0(Y_j)(0)g(Y_j)
\] 
for all $g: \mathscr{S}(d) \to \mathbb{R}^d$. Therefore~\ref{QIP} is equivalent to the following theorem, which is a restatement of Theorem~\ref{Lawler}, \begin{theorem}\label{2}
Let $\mu$ be a stationary ergodic measure on $(\mathscr{S}(d), \mathcal{B}(d)')$. Then under assumptions~\ref{assump:1} and~\ref{assump:2}, there exists a stationary ergodic measure $\lambda$ mutually absolutely continuous to $\mu$ under which $\{Y_j\}$ is stationary and ergodic. 
\end{theorem}
To prove stationarity of the chain is enough to conclude ergodicity via arguments in~\cite{BolthausenSznitman2002}. 
Under~\ref{2}, \[
b= \int_{\mathscr{S}(d)} \mathscr{E(}\omega) d\lambda.
\]
We shall work towards proving Theorem~\ref{2} under assumptions~\ref{assump:1} and~\ref{assump:2}. We shall now introduce the notion of Periodised Environments as mentioned in section~\ref{balanced proof}, fundamental to the proof of~\ref{2}. 
\subsection{Periodised Environments}
Define the equivalence relation $\sim$ by \[
(x_1,x_2,\dots,x_d)\sim (y_1,y_2,\dots,y_d) \iff \frac{(x_i -y_i)}{2n} \in \mathbb{Z}~\text{for each}~i=1,2,\dots,d. 
\]
Let $\mathbb{Z}^d/ \sim ~= T_n$. Clearly $|T_n|=(2n)^d$. We shall also identify $T_n$ as the set in $\mathbb{Z}^d$ given by \[
T_n = \{-n+1,-n+2,\dots,0,1,2,\dots,n\}^d. 
\] 
\begin{definition}[Periodic Environment]
We call $\omega$ a periodic environment if $\omega$ is a function $
\omega:T_n\to S(d)$.  Let $\mathscr{S}_n$ denote the set of all such $\omega$. 
\end{definition}
Fix $\omega \in \mathscr{S}(d)$. We take the periodic environment $(\omega_n, \mathscr{E})$ as \begin{align*}
\omega_n (x) 
&= \omega( [x]_{\sim})\\  
\mathscr{E}(\omega_n) (x) 
&= \mathscr{E} (\omega)( [x]_{\sim}),
\end{align*}
where $x \in \mathbb{Z}^d$. 
 
The corresponding random walk $\{\hat{X}_n\}$ in this periodic environment has the kernel $L_{\mathscr{E}(\omega_n)}$. 
 
For a function $g: T_n \to \mathbb{R}$, we extend it over $\mathbb{R}^d$ as \[
g(x) = g([x]_\sim).
\] 
For an environment $\omega$, denote by $R_n^{\omega}$ the resolvent given by \[
R_n^{\omega} g(x) = \sum_{j=0}^{\infty} (1-\frac{1}{n^2})^j L_{\mathscr{E}({\omega_n})}^j g(x) 
\]
For $g:T_n \to \mathbb{R}$, define the $l_p$($1\leq p\leq \infty)$ norms with respect to the normalized counting measure on $T_n$ by \begin{align*}
||g||_p 
&=\big[\frac{1}{(2n)^d}\sum_{x \in T_n } |g(x)|^p \big]^{\frac{1}{p}}\quad\text{for}~1\leq p <\infty,\\
||g||_{\infty} 
&= \sup_{x \in T_n} |g(x)| . 
\end{align*}
We also define the normed spaces for $1\leq p \leq\infty$ as \[
l^p(T_n) =\{ g: ||g||_p <\infty \}.
\]
As in the original paper by~\cite{Lawler1982}, in the following lemma we state a bound on the resolvent which in turn proves~\ref{2}.

\begin{lemma}\label{3}
There exists $c_1\equiv c_1(d)>0$ such that $\forall \omega \in \mathscr{S}_n$, \[
||R^{\omega}_ng||_{\infty} \leq c_1 n^2||\frac{g(x)}{c(\mathscr{E},x)}||_d. 
\]
\end{lemma}

To show that~\ref{3} $\implies$~\ref{2}, we first state the so called multidimensional ergodic theorem[Chapter 6, Theorem 2.8 of~\cite{Krengel1985}].
\begin{lemma}~\label{4}
    For each $n\geq 1$, define the sequence of measures $\mu_n$, \[
    \mu_n = \frac{1}{(2n)^d} \sum_{x \in T_n} \delta_{\tau_x \omega_n}.
    \]
    Then we have:\begin{enumerate}
  \item $\mu_n \to \mu$ weakly $\mu$-a.e.
  \item For every measurable $f:S(d) \to \mathbb{R}$ such that $f \in L^p(\mu)$ where $p \geq 1$, $\mu_n f \to \mu f$ $\mu$-a.e.
\end{enumerate}
\end{lemma}
First, we see the proof of~\ref{2}, assuming~\ref{3}.

\textit{Proof of Theorem~\ref{2}:} Consider $R_n=R^{\omega_n}_n $ as the following map between normed spaces, \[
R_n: l^d(T_n) \to l^{\infty}(T_n).
\]
Then the dual of $R_n $ is the map\[
R^*_n:l^1 (T_n) \to l^{\frac{d}{d-1}}(T_n).   
\]
Let the invariant measure for the random walk on the group $T_n$ given by $\{[X_{\mathscr{E}(\omega_n)}(k)]_\sim\}_{k \geq 0}$ be $\frac{\phi_n(x)}{(2n)^d}$. Note that $\phi_n $ is the density with respect to the normalised counting measure on $T_n$. Therefore, one can check (or refer to Section 2 of~\cite{guo2012quenched}) the invariant measure for $\{\omega_n(\hat{X}_k)\}_{k \geq 0}$ is given by \[
\lambda_n = \sum_{ x \in T_n} \delta_{\tau_x\omega_n} \frac{\phi
_n(x)}{(2n)^d}. 
\]
Since $\mathscr{S}(d)$ is compact with respect to the product topology, there exists a subsequence $n_k$ such that \[
\lambda_{n_k} \to \lambda~\text{weakly}
\]
for some measure $\lambda$ on $\mathscr{S}(d)$ and $K^{\mathscr{E}}_d$ is invariant under $\lambda$.
Now observe that \[
||R^*_n \phi_n||_{\frac{d}{d-1}}= n^2||\phi_n||_{\frac{d}{d-1}} \leq c_1n^2||\frac{\phi_n(x)}{c(\mathscr{E},x)}||_{1} .
\] By log-convexity of $l_p$ norms and Hölder's Inequality, 
\[
||\phi_n||_{\frac{p}{p-1}}^\frac{p}{d} \leq c_1||\phi_n||_{\frac{p}{p-1}}|| \frac{1}{c(\mathscr{E},x)}||_p.
\]
Therefore
\[
||\phi_n||_{\frac{p}{p-1}} \leq c_1^{\frac{d}{p-d}} || \frac{1}{c(\mathscr{E},x)}||_p^{\frac{d}{p-d}} < \infty.
\]
By Lemma~\ref{4}, \[
\lim_{n \to \infty} || \frac{1}{c(\mathscr{E},x)}||_p^{\frac{d}{p-d}} =\Big [ \int_{\mathscr{S}(d)} c^{-p} d\mu \Big] ^{\frac{d}{p(p-d)}} <\infty \quad \mu\text{-a.e.}
\]
This shows \[
\limsup_{n \to \infty}||\frac{d\lambda_n}{d\mu_n}||_{\frac{p}{p-1}} < C \quad \mu\text{-a.e.}
\]
for some constant $C$. 

Since $\lambda_n$ extends as a measurable function over the entire $\mathscr{S}(d)$, we have that the family of functions \[
\big\{\frac{d\lambda_n}{d\mu_n}\big\}_{n \geq 1}
\]
is uniformly integrable with respect to $\mu$. Therefore by Dunford-Pettis Theorem, $\lambda\ll \mu$ and \[
\int_{\mathscr{S}(d)} |\frac{d\lambda}{d\mu}|^\frac{p}{p-1} d\mu < C^{\frac{p}{p-1}}.
\] 
Referring to arguments from~\cite{BolthausenSznitman2002}, $\lambda$ is ergodic and $\lambda\ll \mu$ and $\lambda\gg\mu$. This concludes the proof.

Thus, in view of the arguments in this section, we shall prove Lemma~\ref{3} in the next section. 
\section{Proof of Lemma~\ref{3}:}
We assume the Uniqueness Principle[Theorem~\ref{Uniqueness}] to be true. We first deduce Lemma~\ref{3} under this assumption and then prove Theorem~\ref{Uniqueness} in the next section.

\textit{Proof of Lemma~\ref{3} :}
Let $\omega \in \mathscr{S}_n$ and let $f: D_n \to [0,\infty)$.
Let \begin{align*}
&\tau_n = \inf \{m: X_{\mathscr{E}(\omega)}(m) \in \partial D_n\} \\
&Qf(x)
= \mathbb{E}_x[ \sum_{k=0}^{\tau_n} f(X_{\mathscr{E}(\omega)}(k))],
\end{align*}
where the expectation is taken assuming $X_{{\mathscr{E}(\omega)}}(0) = x \in D_n^{\circ}$ 

The first part of this proof is dedicated to proving the following lemma, \begin{lemma}~\label{old} For every $f: D_n \to [0,\infty),$
    \[||Qf||_{\infty}\leq c_2 n^2||\frac{f}{c(\mathscr{E},x)}||_{d,D_n}
    \]
    for some $c_2\equiv c_2(d)>1$.
\end{lemma}

Lemma~\ref{old} is equivalent to Lemma~\ref{new} which states that, 
\begin{lemma}~\label{new}
    Let $f: D_n \to [0, \infty)$ be any function. Let $z$ be the concave function guaranteed by Theorem~\ref{Uniqueness}. Then there exists $c_3>1$, dependent only on the dimension such that\[
    ||z||_{\infty} \leq c_3  n^2||\frac{f}{c(\mathscr{E},x)}||_{d,D_n}.
    \]
\end{lemma}

\textit{Proof of Lemma~\ref{new}:}  First we note the equivalence stated above. Recall the~\textit{Monge-Ampère Operator} for a given concave function $z$ given by\[
Mz(x) = \prod_{i=1}^{d} \Delta_i z(x).
\] 
Note that from Theorem~\ref{Uniqueness} one has that \[
\sum_{i=1}^d -\mathscr{E}(x,e_i) \Delta_i z(x)  \geq |Mz|^{\frac{1}{d}} c(\mathscr{E},x) = f(x) 
\]
Therefore, one can show that for every $j\geq1$, \[
\mathbb{E}_x[ z(X_{\mathscr{E}(\omega)}(j \wedge \tau_n)) - z(X_{\mathscr{E}(\omega)}(0))] \leq  \sum_{k=0}^{(j-1) \wedge \tau_n} -f(X_{\mathscr{E}(\omega)}(k)).
\]
Letting $j \to \infty $ completes the proof. This shows that $||Qf||_{\infty} \leq ||z||_{\infty}$ thereby showing the equivalence.

Define the sets for $x \in D_n^{\circ}$,\[
I(x) = \{ a \in \mathbb{R}^d: z(x+e_i) -z(x) \leq a_i \leq z(x) -z(x-e_i), \text{for}~i=1,2,\dots, d\}.
\]
Observe that \[
meas(I(x)) = |Mz(x)|=  \frac{f(x)^d}{c(\mathscr{E},x)^d}, 
\]
where $meas(\cdot)$ is the Lebesgue measure on $\mathbb{R}^d$. 
Now we state an easily provable lemma as in~\cite{Lawler1982}. 
\begin{lemma}[Lemma 8 of~\cite{Lawler1982}]
    Let $a \in \mathbb{R}^d.$ and let $r(x) = a.x +b$ for $b>0$. If $r(x)\geq z(x)$ for every $x\in D_n$ and $\exists x_0 \in D_n ^\circ$ such that $r(x_0)=z(x_0)$, then $a \in I(x_0)$. 

\end{lemma}
Define the set \[
A= \{ x: ||x||_{\infty}\leq \frac{||z||_{\infty}}{4n} \}. 
\]
Then for $b>\frac{3||z||_{\infty}}{2}$, one has \[
a.x+b >||z||_{\infty}>z(x).
\]
Therefore there exists an $b$ such that $a.x+b \geq z(x) $ and $\exists x_0 $ such that $a.x_0 +b= z(x_0)$. Observe that \[
a.x_0 +b = a.(x_0 -x^*) + a.x^*+b \geq \frac{||z||_{\infty}}{2}>0, 
\] where $x^*$ such that $z(x^*) = ||z||_{\infty}$. This shows $x_0 \in D_n^{\circ}$. 
Therefore \[
A \subset \bigcup_{x \in D_n^\circ } I(x),
\] which gives that \[
meas(A) = \frac{||z||^d_{\infty}}{(c_4 n)^d} \leq \sum_{x \in D_n} \frac{f(x)^d}{c(\mathscr{E},x)^d},
\]
where $c_4$ depends on $d$. Therefore,\[
||z||_{\infty}\leq c_4n^2 ||\frac{f(x)}{c(\mathscr{E},x)}||_{d,D_n}.
\]
Choose $c_3= c_4 \vee 2$. This completes the proof of~\ref{new}(c).

Now to bound the resolvent[recall the statement of Lemma~\ref{3}], first we take $m= Kn$ such that $T_n \subset D_m$. We shall specify this $K$. Assume for now that this $K$ only depends on $d$, and it is a natural number.
 
We now bound the resolvent for non-negative $g:T_n\to \mathbb{R}$(and extended as mentioned before) as follows, \[
\mathbb{E}_x\big[\sum_{j=0}^{\infty} \left(1- \frac{1}{n^2}\right) ^j g(X_{\mathscr{E}(\omega_n)}(j)  \big)] \leq  \mathbb{E}_x\big[\sum_{j=0}^{\tau_{Kn}} g(X_{\mathscr{E}(\omega_n)}(j)) \big] + \mathbb{E}_x\big[\sum_{j=\tau_{Kn} +1 }^{\infty} \left(1- \frac{1}{n^2}\right) ^j g(X_{\mathscr{E}(\omega_n)}(j)  \big)],
\]
where $x \in T_n$. 
Using~\ref{old}, one has that \[
\mathbb{E}_x\big[\sum_{j=0}^{\tau_{Kn} } g(X_{\mathscr{E}(\omega_n)}(j)) \big] \leq c_2K^2n^2||\frac{g(x)}{c(\mathscr{E},x)}||_{d,D_m} \leq c'n^2 ||\frac{g(x)}{c(\mathscr{E},x)}||_{d},
\] where $c'\equiv c'(d)$.
The last inequality follows due to the periodic nature of $g$. 

For the other summand, we write it as follows, 
\begin{equation*}
\begin{aligned}
\mathbb{E}_x\Bigg[
    \sum_{j=\tau_{Kn}+1}^{\infty}
    \left(1-\frac{1}{n^2}\right)^j
    g\!\left(X_{\mathscr{E}(\omega_n)}(j)\right)
\Bigg]
&=
\mathbb{E}_x\Bigg[
    \sum_{j=\tau_{Kn}+1}^{\infty}
    \left(1-\frac{1}{n^2}\right)^j
    g\!\left(X_{\mathscr{E}(\omega_n)}(j)\right)
    \mathbf{1}_{\{\tau_{Kn}>n^2\}}
\Bigg]
\\[0.4em]
&\quad +
\mathbb{E}_x\Bigg[
    \sum_{j=\tau_{Kn}+1}^{\infty}
    \left(1-\frac{1}{n^2}\right)^j
    g\!\left(X_{\mathscr{E}(\omega_n)}(j)\right)
    \mathbf{1}_{\{\tau_{Kn}\le n^2\}}
\Bigg].
\end{aligned}
\end{equation*}
The first sum can be bounded as follows, 
\[
\mathbb{E}_x\Bigg[
    \sum_{j=\tau_{Kn}+1}^{\infty}
    \left(1-\frac{1}{n^2}\right)^j
    g\!\left(X_{\mathscr{E}(\omega_n)}(j)\right)
    \mathbf{1}_{\{\tau_{Kn}>n^2\}}
\Bigg] \leq \frac{1}{e} ||R_n g||_\infty .
\]
For the other summand, observe that \begin{align*}
&\mathbb{E}_x\Bigg[
    \sum_{j=\tau_{Kn}+1}^{\infty}
    \left(1-\frac{1}{n^2}\right)^j
    g\!\left(X_{\mathscr{E}(\omega_n)}(j)\right)
    \mathbf{1}_{\{\tau_{Kn}\le n^2\}}
\Bigg] \leq \mathbb{P}_x[\tau_{Kn}\leq n^2] \cdot \|R_n g\|_\infty.\\
\end{align*}
The last line follows from the Strong Markov property since \begin{align*}
&\{\tau_{Kn} \le n^2\} \in \mathcal{F}_{\tau_{Kn}} = \{ A : \forall\, t \in \mathbb{N},\, \{\tau_{Kn} \le t\} \cap A \in \mathcal{F}_{t+1} \} \\
&\sum_{j=\tau_{Kn}+1}^{\infty} \left(1-\frac{1}{n^2}\right)^j g\!\left(X_{\mathscr{E}(\omega_n)}(j)\right) \in \mathcal{F}'_{\tau_{Kn}} = \sigma\!\left( \{ X_{\mathscr{E}(\omega)}(k) : k > \tau_{Kn} \} \right).
\end{align*}
Since $\{|X_{\mathscr{E}(\omega)}(k)|_1^2\}_{k \geq 1}$ is a submartingale, we have by Doob's inequality that \[
\mathbb{P}_x[ \tau_{Kn} \leq n^2] = \mathbb{P}_x [ \max_{0\leq m\leq n^2} |X_{{\mathscr{E}(\omega_n)}}(m)|_1^2 \geq K^2 n^2 ] \leq \frac{n^2 + |x|_1^2}{K^2 n^2}. 
\]
Furthermore, observe that since $x \in T_n$, we must have $|x|^2_1 \leq n^2 d^2$ for every $x \in T_n$. 
Therefore, we have \[
\sup_{x\in T_n} \mathbb{P}_x[ \tau_{Kn} \leq n^2]  \leq \sup_{x \in T_n} \frac{n^2 + |x|_1^2}{K^2 n^2}  = \frac{d^2 +1}{K^2}. 
\]
Now choose $K$ such that \begin{align*}
\frac{d^2 +1 }{K^2} 
&< \frac{1}{2}- \frac{1}{e},\quad T_n \subset D_m. 
\end{align*}
Therefore, combining the deductions we have that \[
||R_n g||_\infty \leq \frac{1}{2} ||R_n g||_\infty + c'n^2||\frac{g(x)}{c(\mathscr{E},x)}||_d \implies ||R_n g||_\infty \leq 2c'n^2||\frac{g(x)}{c(\mathscr{E},x)}||_d.  
\]
By construction, the assumptions of Lemma~\ref{3} are satisfied, thereby completing the proof. 

Therefore, it is now enough to prove Theorem~\ref{Uniqueness}.

\section{Proof of Theorem~\ref{Uniqueness}}
\textit{Proof of Theorem~\ref{Uniqueness}}: We begin by noting that the class $\mathscr{A}(\mathscr{E},f)$ is non-empty. Indeed, consider \[
u(x) = n(n+1) -|x|_1(|x|_1 +1).
\]
Observe that since $\Delta_i u(x) \leq -2$, using ellipticity, one can conclude that $\exists \gamma >0$  such that $\gamma u \in \mathscr{A}$.

We now make two easy observations without proof. The first one is that \[
z_1,z_2 \in \mathscr{A}(\mathscr{E},f) \implies z_1 \wedge z_2 \in \mathscr{A}(\mathscr{E},f). 
\]
The second observation is that \[
z(x)=\inf_{u \in \mathscr{A}(\mathscr{E},f)} u(x) \in \mathscr{A}(\mathscr{E},f). \]
We claim that this function satisfies the second condition of the Theorem. Indeed if it does not then $\exists x_0 \in D_n^{\circ}$ such that \[
|Mz|^{\frac{1}{d}} > \frac{f(x_0)}{c(\mathscr{E},x_0)} \geq0
\]
One can choose $0<\beta < z(x_0)$ such that \[
\Big[ \prod_{i=1}^d (2\beta - z(x+e_i) -z(x-e_i))\Big]^{\frac{1}{d}} > \frac{f(x_0)}{c(\mathscr{E},x_0)}
\]
It is again easy to show that,
\[
z'(x) =
\begin{cases}
\beta, & \text{if } x = x_0, \\
z(x), & \text{otherwise}.
\end{cases}
\]
is a member of $\mathscr{A}(\mathscr{E},f)$. This contradicts the choice of $z(x)$. 
 
\section{Preparation for the proofs of Theorem~\ref{Master}:}
We begin this section by discussing the theory of~\textit{Action of Linear Transformations}. As mentioned earlier, for this proof and from now on, we shall assume $\mathscr{E} = \operatorname{Id}$ unless otherwise indicated.
\subsection{Action of Linear Transformations:}~\label{Action}
Fix $\omega \in \mathscr{S}(d)$. For a given linear transformation $T \in \mathrm{Hom}_{\mathbb{Z}}(\mathbb{Z}^d, \mathbb{Z}^{d})$, the action of $T$ on $\omega = \{\omega(x,v)\}_{x \in \mathbb{Z}^d, v \in V_d}$ is given by 
\[
T.\omega(x,v) := T(\omega)(x,v) = \omega(Tx,v).
\]
Similarly, we also derive the form of the combined action of the shift and the transformation, which we state in the following theorem. 
\begin{theorem}{\label{setup1}}
    $\tau_xT.\omega = T(\tau_{Tx} \omega).$
\end{theorem}
\begin{proof}
    Observe \[(\tau_xT.\omega)(y,v) = T.\omega(x+y,v) = \omega(Tx +Ty,v).\] We also have that \[
    \omega(Tx +Ty,v) = T(\tau_{Tx} \omega)(y,v).
    \]
\end{proof}
This establishes the assertion.
We also describe the action on the transition kernel. We note that we derive these for permutation maps, i.e.,  \[
T(V_d) = V_d. 
\] 
It is straightforward that such a map is invertible.
 
The induced random walk under the transformation $T$ in $\mathbb{Z}^d$ started from $x\in \mathbb{Z}^{d}$ is given by \[
\mathbb{P}_0^{\omega} ( T(X_{\omega}(n)) = T(X_{\omega}({n-1})) + T(e_i)) = \mathbb{P}_0^{\omega}(X_{\omega}(n) =X_{\omega}({n-1}) +e_i) = \omega(X_{n-1},e_i).
\] 
Call the transition kernel of the environment under this new random walk, i.e., the process $\tau_{T(X_{\omega}(n))}\omega$, $K^T_d(\omega, B)$. Then we have the result, 
\begin{theorem}{\label{setup2}}
    $K_d^{\operatorname{Id}} \equiv K_d(T.\omega, T(B)) = K^T_d( \omega, B)$ for every $B \in {\mathcal{B}(d)'}$.
\end{theorem}
\begin{proof}
    Observe that, 
    \[
    K_d^T(\omega,B) 
= \sum_{i=1}^d \Big[ 
    \omega(o, e_i)\,\delta_{\tau_{T(e_i)}\omega}(B)  
    + \omega(o, -e_i)\,\delta_{\tau_{T(-e_i)}\omega}(B) 
\Big] +  \omega(o,o) \delta_{\omega}(B).
    \]
Note that $T.\omega(o,e) = \omega(o,e)$ and by Theorem~\ref{setup1}, \[
\{\tau_{e}{T.\omega} \in T(B)\} = \{T(\tau_{T.e}\omega) \in T(B)\} = \{\tau_{T.e}\omega  \in B\} \quad~\forall e \in V_d.  
\]
This completes the proof.
\end{proof}
This also shows another fact, that is,
\begin{corollary}{\label{setup3}}
    If $\tau_{X_{\omega}(n)} \omega$ has a stationary measure $Q$, $\tau_{T(X_{\omega}(n))} \omega$ is also stationary under $Q\circ T$.  
\end{corollary}
\begin{proof}
 \[
 \int_{\mathscr{S}(d)} K_d^{T}(\omega,B) dQ \circ T = \int_{\mathscr{S}(d)} K_d(T.\omega,T(B)) dQ \circ T= Q\circ T(B) \quad\forall B \in \mathcal{B}({d})'
 \]   
\end{proof}

We shall use these notions fundamentally in our proof. We also note that showing stationarity of the local process is sufficient to exhibit its ergodicity, which we state in the following lemma. We shall prove the lemma by arguments from~\cite{BolthausenSznitman2002}.
\begin{lemma}~\label{equivalence}
    Let $\mathbb{Q}$ be a measure on $(\mathscr{S}(d), \mathcal{B}(d)')$ mutually absolutely continuous to
   $\mu$, such that the local process $\{\pi_0(\tau_{X_{\Lambda(\omega)}(m)}\omega)\}_{m \geq 0}$ is stationary. Further assume $\{\pi_0(\tau_{X_{\Lambda(\omega)}(m)}\omega)\}_{m \geq 0}=\{\pi_0(\tau_{X_{\omega}(m)}\omega)\}_{m \geq 0}$ in distribution under $\mathbb{Q}$. Then $\{\pi_0(\tau_{X_{\omega}(m)}\omega)\}_{m \geq 0}$ is stationary and ergodic under $\mathbb{Q}$.
\end{lemma}
\begin{proof}
Stationarity follows trivially from the assumptions. We just have to show the ergodicity of the local process. 

Let $\mathscr{S}'= S(d)^{\mathbb{N}}$ denote all $S(d)$ valued trajectories of the local process $\{\pi_0(\tau_{X_{\omega}(m)}\omega)\}_{m \geq 0}$.

Let $\theta$ denote the canonical shifts or the $n^{th}$ iterate of the process. Denote by $\tilde{\mathcal{B}}$ the product $\sigma$-algebra. Let $A$ be a set which is invariant under $\theta$, i.e., $\theta^{-1} A = A$. 
 
Let $h(\omega) =\tilde{P}_{\omega}(A) $ where $\tilde{P}_\omega$ is the law if the chain starts from \\$\{\omega(x)\}_{x \in \mathbb{Z}^d}$
We must show 
\[
\tilde{P}_{\mathbb{Q}} (A) = \int_{\mathscr{S}(d)} \tilde{P}_{\omega}(A) d\mathbb{Q}(\omega) \in \{0,1\}
\]
Let $P^*_{\omega}$ denote the law of the local process $\{\pi_0(\tau_{X_{\Lambda(\omega)}(m)}\omega)\}_{m \geq 0}$ starting at $\{\omega(x)\}_{x \in \mathbb{Z}^d}$. 

Therefore, it is enough to show, \[
{P}^*_{\mathbb{Q}} (A) = \int_{\mathscr{S}(d)}P^*_{\omega}(A) d\mathbb{Q}(\omega) \in  \{0, 1\}.
\]

Denote $\bar{\omega}(n) = \tau_{X_{\Lambda(\omega)}(n)} \omega$. 

Start by noting that under the canonical filtration sequence $\mathscr{G}_{k}= \sigma(\omega, \bar{\omega}(1),\dots,\bar{\omega}(k-1))$,  
$h(\bar{\omega}(n))$ is a $\tilde{P}_{\mathbb{Q}}$ martingale since 
\[
\mathbb{E}_{\mathbb{Q}} [ \boldsymbol{1}_{A}| \omega, \bar{\omega}(1),\dots, \bar{\omega}({n}) ] = \mathbb{E}_{\mathbb{Q}} [ \boldsymbol{1}_{A}\circ \theta^n | \omega, \bar{\omega}(1),\dots, \bar{\omega}({n}) ] = h(\bar{\omega}(n)) \quad \mathbb{Q}\text{-a.e.}
\]
By the Levy Upwards Theorem,
\begin{equation}\label{eq:marty}
h(\bar{\omega}(n)) \to \boldsymbol{1}_A \qquad (n \to \infty) \quad\mathbb{Q}\text{-a.e.}
\end{equation}

Next we show $h = \boldsymbol{1}_B$($\mathbb{Q}$ \text{a.e.}) for some $B \in \mathcal{B}(d)'$. If not then \\$\mathbb{Q}( h \in [a,b])>0$ for some $0<a<b<1$. 

But by Birkhoff's Ergodic Theorem, 
\[
\frac{1}{n} \sum_{k=0}^{n-1} \boldsymbol{1}\{h(\bar{\omega}(k))\in [a,b]\} \to \Phi= \mathbb{E}_{\mathbb{Q}}[ \boldsymbol{1}\{h(\omega)\in [a,b]\} | \mathcal{I} ]\quad\mathbb{Q}\text{-a.e.}
\]
where $\mathcal{I}$ is the $\sigma$-algebra of invariant events. 

Observe, 
\[
\mathbb{E}_{\mathbb{Q}} [\Phi] = \mathbb{Q}[h \in [a,b]]>0.
\]
But this contradicts \ref{eq:marty} since $\Phi =0 $ ($\mathbb{Q}$ \text{a.e.}).
Now we note that from the martingale property and the Markov property, and the fact that $\{h(\bar{\omega}(n))\}_{n\geq0}$ is a martingale,
\begin{equation}\label{fund}
\boldsymbol{1}_{B}(\omega) 
  = \mathbb{E}_{\mathbb{Q}}\Big[\,\boldsymbol{1}_{B}(\bar{\omega}(1)) \,\big|\, \omega \,\Big]
  = K_d^{\Lambda} \,\boldsymbol{1}_{B}(\omega)
  \quad \mathbb{Q}\text{-a.e.}
\end{equation}
Combining \ref{fund} with ellipticity,
\begin{equation}{\label{ineq}}
\boldsymbol{1}_{B} \geq \boldsymbol{1}_{B} \circ \tau_e \quad (\mu\text{-\text{a.e.} or } \mathbb{Q}\text{-a.e.}) \quad \forall\, e \in {V}_d.
\end{equation}
Iterating \ref{ineq}, we obtain  
\[
\boldsymbol{1}_{B} \geq \boldsymbol{1}_{B} \circ \tau_x \quad (\mu\text{\text{-a.e.} or } \mathbb{Q}\text{\text{-a.e.}}) \quad \forall\, x \in \mathbb{Z}^d.
\]
Since $\mu[ \omega \in B] = \mu [\tau_{x}\omega\in B]$ for every $x \in \mathbb{Z}^d$, 

\begin{equation}{\label{erg}}
 \boldsymbol{1}_{B} = \boldsymbol{1}_{B} \circ \tau_x \quad \mu\text{\text{-a.e.} or}~\mathbb{Q}\text{\text{-a.e.}}
\end{equation}
Consider the set \[
\tilde{B} = \cap_{x \in \mathbb{Z}^d} \tau_{-x}B .
\] 
Observe $(\tau_x)^{-1} \tilde{B} = \tilde{B}$ which shows $\tilde{B}$ is invariant under shifts. 

From \ref{erg} we get that $ \mu(\tilde{B})= \mu(B) = 0$ or $1$ (by ergodicity). From the mutual absolute continuity of $\mathbb{Q}$ and $ \mu$, 
$\mathbb{Q}(B) =P^*_{\mathbb{Q}}(A)=0$ or $1$. This proves ergodicity of the local process in the environment $\mathbb{Q}$.
\end{proof}
Thus, in view of the previous proof, we have reduced the proof to the following lemma. 
\begin{lemma}~\label{Final}
    Let $\mathbb{Q}$ be a measure on $(\mathscr{S}(d), \mathcal{B}(d)')$ mutually absolutely continuous to
   $\mu$, such that the local process $\{\pi_0(\tau_{X_{\Lambda(\omega)}(m)}\omega)\}_{m \geq 0}$ is stationary and ergodic. Then $\{\pi_0(\tau_{X_{\Lambda(\omega)}(m)}\omega)\}_{m \geq 0}=\{\pi_0(\tau_{X_{\omega}(m)}\omega)\}_{m \geq 0}$ in distribution (and therefore both are stationary sequences) under $\mathbb{Q}$.
\end{lemma}

\section{Proof of Lemma~\ref{Final}:} 
Let $\lambda'$ be the stationary ergodic measure of the chain $\{\tau_{X_{\Lambda(\omega)}(n)} \omega\}_{n \geq 0}$. We set $\mathbb{Q} = \lambda'$. Fix $\omega \in \mathscr{S}(d)$. 

Recall that $\{X_{\omega}(n)\}_{n \geq 0}$ is the random walk with the law \begin{align}
    \mathbb{P}[X_{\omega}(1) =x+v \Big |X_{\omega}(0) =x] = \omega(x,v) \quad v \in V_d.
\end{align}
The environment process is therefore given by $K_d^{\operatorname{Id}}\equiv K_d$. 

Define the reflected walk, $\{X_{\alpha(\omega)}(n)\}_{n \geq 0}$ with the law \begin{align}
    \mathbb{P}[X_{\alpha(\omega)}(1) =x+v \Big |X_{\alpha(\omega)}(0) =x] = \omega(x,-v). 
\end{align} 
The environment process is given by \[
G_d(\omega,B):=K^{\alpha}_d(\omega, B) = K_d(\alpha.\omega, \alpha.B) \quad \forall B \in \mathcal{B}(d)'.
\]
We redefine $\{X_{\Lambda(\omega)}(n)\}_{n \geq 0}$ in the following manner.\[
L_{\Lambda(\omega)} g(X_{\Lambda(\omega)}(n)) = \mathbb{E}[\gamma_n L_{\omega} g(X_{\Lambda(\omega)}(n)) + (1-\gamma_n) L_{\alpha(\omega)} g(X_{\Lambda(\omega)}(n))] \quad \forall g: \mathbb{Z}^d \to \mathbb{R},~\gamma_n \sim \text{Ber}\left(\frac{1}{2}\right)
\]
where $L_{\Lambda(\omega)}, L_{\omega}, L_{\alpha(\omega)}  $ kernels of $ X_{\Lambda(\omega)}, X_{\omega}, X_{\alpha(\omega)}$ respectively and $\gamma_{n}$ is a Bernoulli Random Variable independent of the environment and position of the walk. Further $\{\gamma_n\}_{n \geq 1}$ is a sequence of i.i.d. Bernoulli random variables.    
We therefore define the conditional environment process \[
K_{d}^{\Lambda}(\omega, B) \Big| \gamma := K_\gamma(\omega, B) = \gamma K_d(\omega, B) + (1-\gamma) G_d(\omega, B).
\]
We now state an easy lemma without proof. 
\begin{lemma}~\label{main*}
    $K_\gamma(\alpha. \omega , B) = K_{1-\gamma}(\omega, \alpha. B)$.
\end{lemma}
This is the control argument using Bernoulli variables as we had mentioned, and now we shall use it explicitly to demonstrate the stationarity of the local process or the equivalent statement of Lemma~\ref{Final}.  Let $\bar{\omega}(k)$ denote the environment process $\tau_{X_{\Lambda(\omega)}(k)}\omega$ for $ k \geq0$.
\begin{lemma}~\label{main**} 
Let $ n \geq 1$. Let $B_0, B_1,\dots,B_n \in \mathcal{B}(d)' $. Let the probability measure associated with $\omega$ be $\lambda'$. Then for every $0\leq k \leq n-1$, we have\begin{align*}
&\lambda'\Bigl[
    \omega \in B_0,\,
    \bar{\omega}(1) \in B_1,\,
    \ldots,\,
    \bar{\omega}(k) \in \alpha(B_{k}),\,
    \ldots,\,
    \bar{\omega}(n) \in B_n
    \,\Bigm|\,
    \gamma_1,\ldots,\gamma_n
\Bigr] \nonumber \\
&\qquad =
\lambda'\Bigl[
    \omega \in B_0,\,
    \ldots,\,
    \bar{\omega}(k) \in \alpha(B_{k}),\ \bar{\omega}(k+1) \in \alpha(B_{k+1}),
    \ldots,\,
    \bar{\omega}(n) \in B_n
    \,\Bigm|\,
    \gamma_1,\ldots,1-\gamma_{k+1},\ldots,\gamma_n
\Bigr].
\end{align*}

\end{lemma}
To see how this implies stationarity, take $B_i = \pi_0^{-1}(A_i)$ where $A_i \in \mathcal{B}(d)$. Then \begin{align*}
\lambda' \left[ \bigcap_{k=0}^n \{ \bar{\omega}(k) \in \pi_0^{-1}(A_k) \} \,\middle|\, \gamma_1,\dots, \gamma_n \right] = \frac{1}{2^n} \sum_{\{\gamma_i\}_{i=1}^n \in \{0,1\}^n} \lambda' \left[ \bigcap_{k=0}^n \{ \bar{\omega}(k) \in \pi_0^{-1}(A_k) \} \,\middle|\, \{\gamma_i\}_{1\leq i\leq n} \right]
\end{align*}
The last expression reduces to $\lambda'[ \omega \in \pi_0^{-1} (A_0),\dots,\bar{\omega}(n) \in \pi_0^{-1}(A_n)]$. Since $\{\bar{\omega}(k)\}_{k \geq 0}$ is stationary under $\lambda'$, this establishes the assertion of Lemma~\ref{Final} and therefore stationarity of $\{\pi_0(\tau_{X_{\omega}(k)} \omega) \}_{k \geq 0}$. Note that taking $\gamma_i=1$, it is clear that \[
\tau_{X_{\Lambda(\omega)}(n)} \omega |\{\gamma_i=1\quad i=1,2,\dots,n\} = \tau_{X_{\omega}(n)} \omega. 
\]
This shows that choosing $\gamma_i$'s controls the path suitably. 

Therefore, it is now enough to establish Lemma~\ref{main**}. 

\textit{Proof of Lemma~\ref{main**}:} We shall prove this result using induction. First for $n=1$, we note that \[
\lambda'[\alpha(\omega) \in B_0,\bar{\omega}(1) \in B_1|\gamma] = \int_{\alpha(B_0)} K_{\gamma}(\omega,B_1) d\lambda'(\omega).
\]
We perform the substitution $\omega \mapsto \alpha(\omega)$. The integral reduces to \[
\int_{B_0} K_{\gamma} (\alpha.\omega, B_1) d \lambda' \circ \alpha .
\]
By Lemma~\ref{main*}, the integral reduces to \[
\int_{B_0} K_{1-\gamma} (\omega, \alpha.B_1) d \lambda' \circ \alpha = \mathbb{E}[\boldsymbol{1}_{\omega} (\alpha(B_0)) \mathbb{E}[\boldsymbol{1}_{\bar{\omega}(1)}(\alpha.B_1)\Big|\omega,1-\gamma]]
\]
This reduces to $\lambda'[\omega\in \alpha.B_0, \bar{\omega}(1)\in \alpha.B_1| 1-\gamma]$. 

Suppose the result is true for $n=m\geq1$. Let $n=m+1$. We verify the statement for $k=m+1$ only for simplicity, as others follow from the technique of proof. For $k=m+1$, we write \[
\int_{B_0 \times B_1 \times B_2 \times \dots \times \alpha.B_{m}} K_{\gamma_{m+1}}(\omega_m, B_{m+1}) d \lambda'( \omega_m,\dots, \omega| \gamma_1,\dots, \gamma_m) .
\]  
Perform the substitution $ (\omega_m,\dots, \omega) \mapsto (\alpha. \omega_m,\dots, \omega)$ to get \[
\int_{B_0 \times B_1 \times B_2 \times \dots \times B_{m}} K_{ \gamma_{m+1}}(\alpha.\omega_m, B_{m+1}) d \lambda'( \alpha.\omega_m,\dots, \omega| \gamma_1,\dots, \gamma_m) .
\]
Using Lemma~\ref{main*}, the integral reduces to 
\begin{align}
\int_{B_0\times B_1 \times B_2 \times \cdots \times B_m}
K_{1-\gamma_{m+1}}\bigl(\omega_m, \alpha(B_{m+1})\bigr)\,
d\lambda'\bigl(
    \alpha(\omega_m),\ldots,\omega
    \,\Bigm|\,
    \gamma_1,\ldots,\gamma_m
\bigr) \nonumber.
\end{align}
Now we write the above expression as \[
\mathbb{E}\Big[\mathbb{E}\Big[\boldsymbol{1}_{\alpha. B_{m+1}}(\bar{\omega}(m+1))\Big| \bar{\omega}(m),1-\gamma_{m+1}\Big] \boldsymbol{1}_{\bar{\omega}(m)}(\alpha.B_m)\boldsymbol{1}_{\bar{\omega}(m-1)}(B_{m-1})\dots, \boldsymbol{1}_{\omega}(B_{0}) \Big| \gamma_1,\dots, \gamma_m \Big].
\]
This reduces by definition to the desired form, which completes the proof of Lemma~\ref{main**}.

\bibliographystyle{aomalpha}
\bibliography{bib,references}
\end{document}